\documentclass{article}[11pt]

\usepackage{amssymb}
\usepackage{verbatim}
\usepackage{graphicx}

\usepackage{xcolor}

\newcommand{\JZR}[1]{{\color{black} #1}} 
\newcommand{\JZ}[1]{{\color{black} #1}} 
\newcommand{\DR}[1]{{\color{black} #1}}

\newtheorem{thm}{Theorem}[section]
\newtheorem{prop}[thm]{Proposition}

\newtheorem{cor}[thm]{Corollary}

\newtheorem{lem}[thm]{Lemma}

\newcommand{\SB}[1]{   {\color{black} #1}} 

\def\UBi{{\rm B1} } 
\def\UB2{{\rm B2} }

\begin{document}

\title{The 2-rainbow domination number of Cartesian product of cycles}{}
 
\author{Simon Brezovnik$^{a,b}$ 
\and 
Darja Rupnik Poklukar$^{a}$
\and
Janez \v{Z}erovnik$^{a,c}$
}

\date{\today}

\maketitle

\begin{center}
$^a$ Faculty of Mechanical Engineering, University of Ljubljana, A\v{s}ker\v{c}eva 6,  Ljubljana 1000, Slovenia (simon.brezovnik@fs.uni-lj.si, darja.rupnik@fs.uni-lj.si, janez.zerovnik@fs.uni-lj.si)\\
$^b$ Institute of Mathematics, Physics and Mechanics, Ljubljana, Slovenia\\
$^c$ Rudolfovo - Science and Technology Center, Novo mesto, Slovenia\\
\end{center}

\textbf{Submitted to Ars Mathematica Contemporanea.}

\begin{abstract} 
A {\em $k$-rainbow dominating function} ($k$RDF) of $G$ is a function that assigns subsets of $ \{1,2,...,k\}$ to the vertices of $G$ such that for vertices $v$ with $f(v)=\emptyset $ we have $\bigcup\nolimits_{u\in N(v)}f(u)=\{1,2,...,k\}$.
 The {\em weight} $w(f)$ of a $k$RDF $f$ is defined as $w(f)=\sum_{v\in V(G)}\left\vert f(v)\right\vert $.
The minimum weight of a $k$RDF of $G$ is called the {\em $k$-rainbow domination number} of $G$, which is denoted by $\gamma_{rk}(G)$.
In this paper, we study the 2-rainbow domination number of the Cartesian product of two cycles.
Exact values are \JZR{given} for a number of infinite families and we prove lower and upper bounds for all other cases. \\
\end{abstract}



\section{Introduction}
The Cartesian product is one of the standard graph products \cite{Imrich}.
For example, meshes, tori, hypercubes and some of their generalizations are Cartesian products.

Graph domination is one of the most popular topics in graph theory \cite{Haynes,Haynes2,Book2020}.
There are many variants motivated by interesting applications.
The $k$-rainbow domination problem was first studied in \cite{Bressar2008} and has attracted a lot of attention.
For example, in \cite{Bressar20072394}, the authors proved that the concept of 2-rainbow domination is equivalent to ordinary domination in the prism $G\Box K_2$ and established the NP-completeness of determining whether a graph has a 2-rainbow dominating function with a certain weight. Furthermore, in \cite{Bresar20151263} the authors characterize the pairs of graphs $G$ and $H$ for which
${\gamma(G \Box H)} = \min\{V (G), V (H)\}$. There are also many papers that observe 2-rainbow domination on generalized Petersen graphs, for example \cite{clanek1, Tong20091932, Xu20092570, Erves2021-2}. In recent years, research on the 2-rainbow domination and its variants has expanded even further. For example, in \cite{Kuzman2020454} the $k$-rainbow domination on regular graphs was investigated. Meybodi et al.~\cite{Meybodi2021277} investigated $k$-rainbow domination in graphs with bounded tree-width. In \cite{Kim20213447} Kim investigated $k$-rainbow domination in middle graphs in the context of operations research. In \cite{Brezovnik2019214} an independent variant of $k$-rainbow domination on the lexicographic products of graphs was investigated. Recently, Kosari and Asgharsharghi \cite{Kosari2023} studied the $l$-distance $k$-rainbow domination numbers of graphs.
 For further references, see \cite{Bressar2020}.


\bigskip

In this paper we study 2-rainbow domination numbers of the Cartesian product of two cycles.
We provide exact values for a number of infinite families and prove lower and upper bounds for all other cases. Our main results are summarized in the following two theorems.

For $n \equiv 0 \pmod 6$ the first theorem gives exact values of $\displaystyle \gamma_{r2}(C_{m} \Box C_n)$
for $m \equiv 0,2 \pmod 3$ and bounds  with   gap  at most $\frac{1}{2}n$ for the case  $m \equiv 1 \pmod 3$.

\begin{thm}\label{Main} \label{AllBounds} 
Let $m \geq 3$ and $n \ge 6$, $n \equiv 0 \pmod 6$.
Then we have 
 \begin{enumerate}
\item[a)] 
      if   $ m \equiv 0 \pmod 3  $      then  $\displaystyle \gamma_{r2}(C_{m} \Box C_n)   =    \frac{m}{3}  n   $ .   
\item[b)] 
      if   $ m \equiv 1 \pmod 3 $    then  
$$\displaystyle    \left(\frac{m-1}{3} + \frac{1}{2}\right)  n  \le  \gamma_{r2}(C_{m} \Box C_n)   \le  \frac{m+2}{3}    n    \,  .$$   
\item[c)] 
      if   $ m \equiv 2 \pmod 3  $   then  $\displaystyle \gamma_{r2}(C_{m} \Box C_n)   =    \frac{m+1}{3}    n   $ .   
\end{enumerate}    
\end{thm} 

The second theorem is a summary of the lower and upper bounds of the products of cycles, covering all cases.
Note that the gap is at most $\frac{1}{2}n + 2 \left\lceil \frac{m}{3}\right\rceil $.

\begin{thm} \label{AllBounds2}
Let $m \geq 3$ and $n \ge 6$.
Then
$$\displaystyle \left(\left\lfloor \frac{m}{3}\right\rfloor + \alpha \right) n \le \gamma_{r2}(C_{m} \Box C_n)
\le
\JZ{\min}
 \left\{ \left\lceil \frac{m}{3}\right\rceil ( n + \beta), \left\lceil \frac{n}{3}\right\rceil (m + \gamma) \right\} \,,
$$
where
$\displaystyle\alpha =
 \left \{ \begin{array}{ll}
 0, & m \equiv 0 \pmod 3\\
 \frac{1}{2} & m \equiv 1 \pmod 3\\
 1, & m \equiv 2 \pmod 3 \\
 \end{array} \right. \,,
$
$\displaystyle\beta =
 \left \{ \begin{array}{ll}
 0, & n \equiv 0 \pmod 6\\
 1, & n \equiv 1,2,3,5 \pmod 6\\
 2, & n \equiv 4 \pmod 6\\
 \end{array} \right. \,,
$
\\
and
$
 \displaystyle\gamma =
 \left \{ \begin{array}{ll}
 0, & n \equiv 0 \pmod 6\\
 1, & n \equiv 5 \pmod 6\\
 2, & n \equiv 1,2,3,4 \pmod 6\\
 \end{array} \right. \,.
$
\end{thm}
%
\smallskip

The upper bounds are given in alternative form as \JZR{Corollary} \ref{CorTable}. The rest of the paper is organized as follows.
In the next section we recall some basic definitions and some useful previously known results.
In Section 3 we prove lower bounds.
In Section 4, we study two patterns that allow constructions that yield upper bounds.
The final section contains a number of ideas for future research.

\section{Preliminaries}
 
A finite, simple and undirected graph $G=(V(G),E(G))$ is given by a set of vertices $V(G)$ and a set of edges $E(G)$. 
 As usual, the edges $\{i,j\} \in E(G)$  are shortly denoted  by $ij$. 
  
 A set $S$ is a dominating set if every vertex in the complement $V(G) \setminus  S$ is adjacent to a vertex in $S$. The minimum cardinality of a dominating set of $G$ is called the  domination number $\gamma (G)$.

The Cartesian product of two graphs, $G\Box H$, is the graph with vertex set $V(G)\times V(H)$, in which two vertices are adjacent if and only if they are equal in one coordinate and adjacent in the other.
The Cartesian product of graphs is one of the standard graph products \cite{Imrich}. The Cartesian product is commutative.
In other words: $C_m\Box C_n$ is isomorphic to $C_n\Box C_m$.
So if we consider the product of the cycles $C_m\Box C_n$, we can assume $m \le n$.

For a given vertex $v\in V(G)$, the open neighborhood $N(v)$ consists of the vertices adjacent to $v$. The degree of vertex $v$ equals $\textnormal{deg}_G(v)=|N(v)|.$ The minimum and the maximum degree of a graph $G$ are denoted by $\delta(G)$ and $\Delta(G).$

Let $f$ be a function that assigns to each vertex a set of colors chosen from the set $\{1,2,\dots,k\}=[k]$, with the property that for each $v\in V(G)$ with $f(v)=\emptyset$
we have
$$\bigcup_{u\in N(v)} f(u)=[k].$$
Such a function $f$ is called a $k$-rainbow dominating function (kRDF) of $G$. The weight of $f$, \JZR{denoted by $w(f)$, is defined as}
$$ w(f)=\sum_{v\in V(G)}|f(v)|.$$
Recall that $f(v)$ is a set of colors and $|f(v)|$ denotes the number of elements in $f(v)$.
The minimum weight of a kRDF on $G$ is called the $k-$rainbow domination number of $G$, $\gamma_{rk}(G)$ \SB{and in this case the function is called\textit{ $\gamma_{rk}(G)$-function}}. It is clear that for $k=1$ this definition corresponds to the usual domination.

The following theorems, which connect rainbow domination with (ordinary) domination, will be of interest here.

\begin{thm}\cite{Bressar2008}
For any graph $G$  we have
$
\gamma_{rk}(G)=\gamma (G\Box K_{k}).
$
\end{thm}
\begin{thm}\cite{BertL2004}\label{T:I2}
For any graph $G$  we have
$
\gamma_{rk}(G) \leq k\gamma (G).
$
\end{thm}

In \cite{Klavzzar1995129}, it was shown that  $\gamma(C_3 \Box C_n)=n- \lfloor \frac{n}{4}\rfloor$, $\gamma(C_4 \Box C_n)=n$ for $n \geq 4$ and 
$$\gamma(C_5 \Box C_n)=\left \{ \begin{array}{ll}
              n, &   n \equiv 0 \pmod 5\\
              n+1, &  n  \equiv 1,2,4 \pmod 5 \\
             \end{array} \right. 
             \,.$$ 
This result was supplemented  in \cite{Mohamed2}, where it was shown that 
$\gamma(C_5 \Box C_n)=n+2$ for $n\equiv 3 \pmod 5$ 
and also exact values for $\gamma(C_6 \Box C_n)$ and $\gamma(C_7 \Box C_n)$ were given.
 In \cite{Mohamed} it was proved that
$\lceil \frac{9n}{5} \rceil \leq \gamma(C_8 \Box C_n) \leq \lceil \frac{9n}{5} \rceil +1$ for $n \geq 8$ and exact value for $\gamma(C_9 \Box C_n)$ was given.

Considering 2-rainbow domination number of the Cartesian product of two cycles, the well-known inequality is (see \cite{Stepien2014748})
$$\frac{mn}{3}\le \gamma_{r2}(C_m\Box C_n)\le 2 \gamma (C_m\Box C_n).
$$
The 2-rainbow domination number of the products $C_3 \Box C_n$ and $C_5 \Box C_n$ were studied in \cite{Stepien2014113,Stepien2014748}.
In \cite{Stepien2015668} a complete characterization of graphs $C_m \Box C_n$ was given, for which the 2-rainbow domination number is equal to $\frac{mn}{3}$.
A summary of the then known results on the $k$-rainbow domination of the Cartesian product of cycles appears in \cite{Gao2019}.
In the following 
we recall the previously known formulas for 2-rainbow domination numbers for $C_m\Box C_n$.
 
\medskip
 \noindent
\begin{tabular}{ll}
 \hline \\
 Result &   Ref. \\
  \hline  
$\gamma_{r2}(C_{3} \Box C_n)    =
        \left \{ \begin{array}{ll}
              n, &   n \equiv 0 \pmod 6\\
              n+1, &  n  \equiv 1,2,3,5 \pmod 6 \\
              n+2, &  n  \equiv 4 \pmod 6 \\
             \end{array} \right.$ &   \cite{Stepien2014748} \\
  \hline \\
 $\gamma_{r2}(C_{4} \Box C_n)   =
        \left \{ \begin{array}{ll}
              \left\lfloor\frac{3n}{2} \right\rfloor, &  n  \equiv 0  \pmod 8 \\[1mm]
              \left\lfloor\frac{3n}{2} \right\rfloor+1, &  n  \equiv 2,4,5  \pmod 8 \\[1mm]
              \left\lfloor\frac{3n}{2}\right\rfloor+2 , &  n  \equiv 1,3,6,7  \pmod 8\\[1mm]
             \end{array} \right.
$ &  \cite{NatAcademy}  \\
   \hline \\
 $\gamma_{r2}(C_5 \Box C_n)=2n$ &   \cite{Stepien2014113} \\
   \hline \\ 
 $\gamma_{r2}(C_8 \Box C_n)=3n$ &   \cite{NatAcademy}   \\
   \hline \\
$\displaystyle \gamma_{r2}(C_m \Box C_n)= \frac{mn}{3}$, if and only if either \\
 $m \equiv 0 \pmod 3, n \equiv 0 \pmod 6$ or
$m \equiv 0 \pmod 6, n \equiv 0 \pmod 3$ &\cite{Stepien2015668}\\ 
   \hline 
\end{tabular}

\section{Lower bounds for 2-rainbow domination of $C_m \Box C_n$}

For simplicity, we introduce some more notations.
The vertices of $V(C_m\Box C_n)$ are denoted by $(i,j)$ for $i\in [m]$ and $j\in [n]$.
The coordinates $i$ and $j$ are taken modulo $m$ and $n$ respectively, so that we identify $m$ and $0$, for example.
For a fixed (small) $m$, the set of vertices is
$\mathcal{C}^i=\{(i,1),(i,2),\cdots,(i,n)\}$, $i\in [m]$ is called the $i$-th column of $C_{m}\Box C_{n}$.

Let $f$ be a 2RDF of $C_m \Box C_n$ and $s_i = \sum_{x \in \mathcal{C}^i} |f(x)|$.
The sequence
$(s_{1}$, $s_{2}$, $\cdots$, $s_{m})$,
is called the 2RDF sequence that corresponds to $f$.
We also use $f(i,j)=f(v)$ to denote the value of $f$ at vertex $v= (i,j)$ for $i \in [m]$ and $j \in [n]$.


First, we recall a general bound for regular graphs. We believe that it is well known, although we have not found a reference with a proof.
Therefore, for the sake of completeness, we provide a short proof.

\begin{lem}\label{GeneralLB}
Let $G$ be an $r$-regular graph. Then
$\displaystyle \gamma_{rk}(G) \ge \frac{k}{r+k} |V(G)| \,.
$
\end{lem}
\begin{proof} \JZR{
Assume that $f$ is a $k$RDF and that $n^*$ vertices are colored.
Then double count to obtain $r w(f) \ge (|V(G)| - n^* )k$.
Apply $n^* \le w(f) $ and the conclusion follows.}
\end{proof}
%

Cartesian products of cycles are 4-regular graphs, and we consider  2-rainbow domination, so we need a special case of Lemma \ref{GeneralLB}, namely $k=2$ and $r=4$.
\begin{cor}\label{GeneralLB2}  
Let $G$ be a 4-regular graph. Then  
$\displaystyle \gamma_{r2}(G)  \ge \frac{1}{3}  |V(G)| \,.
$    
\end{cor}
\JZR{Note that the statement also follows from \cite[Lemma 2.2, Case (6)]{Kuzman2020454}.}

%
The next lemma will be useful to obtain better lower bounds for Cartesian products of cycles.
In particular, for bounds of $\gamma_{r2}\left( C_{m}\Box C_n\right)$.
Recall that $s_i = \sum_{x \in \mathcal{C}^i} |f(x)|$.

 \begin{lem} \label{L}
Let $f$ be a $\gamma_{r2}\left( C_{m}\Box C_n\right)$-function.
Write $m = 3k + \ell$, where $\ell \equiv m \pmod 3 $.
%
Then
\begin{enumerate}
\item[a)]
$s_{i-1}+s_{i+1} \geq 2m - 4s_{i} = 6k + 2\ell - 4s_{i} \,,
$
\item[b)] \JZR{if} 
$k \ge s_{min} =\min \{ s_{i-1}, s_{i+1}\}$,
 then
$$ s_{max} \geq 2m - 4s_{i} -s_{min} \ge 5k + 2\ell - 4s_{i} \, ,
$$
where $s_{max} = \max \{ s_{i-1}, s_{i+1}\}$.
\end{enumerate}
\end{lem}

\begin{proof}
 Note that at most $s_{i}$ vertices of the column $\mathcal{C}^i $ are colored
(this holds in the case when all $|f(v)| =1$).
Other ( uncolored ) vertices in $\mathcal{C}^i $, at least $m- s_{i}$ of them,
have a total demand at least $2(m-s_{i})$.
Since at most $2s_{i}$ of this demand can be fulfilled by the colored vertices of $\mathcal{C}^i $,
we must have at least $2m- 4 s_{i}$ colors in the neighborhood of $\mathcal{C}^i $.
Equivalent to this is $s_{i-1}+s_{i+1} \geq 2m - 4s_{i}$.
So if we use $m = 3k + \ell$, we have
$$s_{i-1}+s_{i+1} \geq 2m - 4s_{i} = 6k + 2\ell - 4s_{i} \, ,
$$
as required. Finally, if
 $ k \ge s_{min} =\min \{ s_{i-1}, s_{i+1}\}$,
then
$$ s_{max} \geq 2m - 4s_{i} - s_{min} = 6k + 2\ell - 4s_{i} - s_{min} \ge 5k + 2\ell - 4s_{i} \, ,
$$
and the proof is complete. \end{proof}

The next observation provides lower bounds.
The proof is based on the discharging argument and follows the ideas of \cite{NatAcademy} and \cite{IEEE}.

\begin{prop}\label{c4cnLowerNew}
Let $m \geq 3$  and $n \ge 3$.
Write $m = 3k + \ell$, where $\ell \equiv m \pmod 3 $.
Then  $$\gamma_{r2}(C_{m} \Box C_n)  \ge  kn +  \ell  \frac{n}{2}=\frac{mn}{3}+\ell \frac{n}{6} \,.$$
\end{prop}
\begin{proof}
\SB{Note that when $m = 3k$, the proof follows directly from Lemma \ref{GeneralLB} In the following, we write the proof for the case when $m = 3k + 1$, since the proof for the case when $m = 3k + 2$ is similar and can therefore be omitted.

 \SB{Let $f$ be a $\gamma_{r2}$-function on the vertex set of $C_{m} \Box C_n$ and let $(s_1, s_2, \ldots , s_m)$,
be the 2RDF sequence corresponding to $f$. } We define a discharging rule in which
 the columns with sufficiently large $s_{i}$ give half of their overweight to \SB{one or both of the neighboring columns}. For this purpose, let $f'$ be a function on the vertex set of $C_{m} \Box C_n$ that assigns a positive real number to each vertex. Denote by $s_i' = \sum_{x \in \mathcal{C}^i} f'(x)$ and let $(s_1', s_2', \ldots , s_m')$ 
be the sequence corresponding to $f'$. Moreover, we define $f'$ such that the following holds:}

If 
  $s_{i} > k + \frac{1}{2} $  then set  $s'_i =  k + \frac{1}{2}$. \SB{If $s_i \le k + \frac{1}{2}$, then
\begin{itemize}
    \item if $s_{i-1}>k + \frac{1}{2}$ and $s_{i+1}>k + \frac{1}{2}$, then $s'_i = s_{i}  +  \frac{1}{2} (s_{i-1} -(k+\frac{1}{2}))+  \frac{1}{2} (s_{i+1} -(k+\frac{1}{2})) $,
    \item if $s_{i-1}>k + \frac{1}{2}$ and $s_{i+1}<k + \frac{1}{2}$, then $s'_i = s_{i}  +  \frac{1}{2} (s_{i-1} -(k+\frac{1}{2}))$,
     \item if $s_{i-1}<k + \frac{1}{2}$ and $s_{i+1}>k + \frac{1}{2}$, then $s'_i = s_{i}  +  \frac{1}{2} (s_{i+1} -(k+\frac{1}{2}))$.
\end{itemize}}

We claim that  $s'_i \ge k + \frac{1}{2}$ for all $i$.
Assume $s_{i} \le k + \frac{1}{2}$. 
Note that, since $s_i$ is an integer, $s_{i} \le k + \frac{1}{2}$ implies  $s_{i} \le k  $.  
Again,   if   $s_{i-1} >k$ and $s_{i+1} >k$ then,   by Lemma \ref{L},  
\begin{eqnarray} 
s'_i &\ = & s_{i} + \frac{1}{2}(s_{i-1} -(k + \frac{1}{2})) + \frac{1}{2}(s_{i+1} -(k + \frac{1}{2}))    \nonumber \\
  &=& s_{i} + \frac{1}{2}(s_{i-1}   +  s_{i+1}) -  (k + \frac{1}{2})  \nonumber \\
   &\ge&   s_{i} + 3k +1  -2s_{i} -k  - \frac{1}{2}   = 2k  +\frac{1}{2} - s_{i}  \nonumber \\
     &=&   k  +\frac{1}{2} +(k - s_{i})    \ge     k  +\frac{1}{2}   \,.  \nonumber 
\end{eqnarray}
 or,  when  $s_{min} =\min \{ s_{i-1}, s_{i+1}\} \le k$,   
\begin{eqnarray} 
s'_i &=& s_{i} + \frac{1}{2}(s_{max} -(k + \frac{1}{2}))    \nonumber \\ 
   &\ge&   s_{i} +  2k +1  -2s_{i}    - \frac{1}{4}   = 2k  +\frac{3}{4} - s_{i}       \,.  \nonumber 
\end{eqnarray} 
Recall that $s_{i}$ is an integer,  so $s_{i} \le  k + \frac{1}{2} $ is equivalent to $s_{i} \le  k   $, and hence 
$$s'_i = 2k  + \frac{3}{4} - s_{i}   =    k  +\frac{3}{4}   + (k- s_{i})   >   k + \frac{1}{2},
$$ 
which   implies  $\gamma_{r2}(C_{m} \Box C_n)  = \sum_i s_{i} \ge \sum_i  s'_i \ge n(k + \frac{1}{2}) $.\\

\noindent Summarizing, we get   
\begin{enumerate}
\item[a)] 
      $\gamma_{r2}(C_{m} \Box C_n)  \ge  kn $ when $\ell = 0 $, 
\item[b)]   
   $\gamma_{r2}(C_{m} \Box C_n)  \ge  kn + \frac{ n }{2} $  when   $\ell = 1 $, and 
\item[c)]  $\gamma_{r2}(C_{m} \Box C_n)  \ge  kn + n$ when  $\ell = 2 $.  
\end{enumerate}      
which in turn implies
 $$\gamma_{r2}(C_{m} \Box C_n)  \ge  kn +  \ell  \frac{n}{2}=\frac{mn}{3}+\ell \frac{n}{6} $$
as claimed.
  \end{proof}

\section{Upper bounds}

Recall the characterization of the products where the general lower bound is attained \cite{Stepien2015668}.
More precisely, the result is given in the next theorem.

\begin{thm}\label{Exact}    \cite{Stepien2015668}
If   either  $m \equiv 0 \pmod 3$   and $n \equiv 0 \pmod 6$,
or   $m \equiv 0 \pmod 6$   and $n \equiv 0 \pmod 3$, then
$$\gamma_{r2}(C_{m} \Box C_n)   = \frac{1}{3} mn \,.
$$    
\end{thm}

For later reference, observe    such  2RDF  may be  based on the pattern
\begin{equation}
\left[\begin{array}{ccccccccccc}
 \dots &\dots &\dots &\dots &\dots &\dots &\dots &\dots&\dots &\dots&\dots  \\
\dots & 1& 0& 0& 2& 0& 0& 1  &  0  &  0 &\dots  \\
\dots & 0& 2& 0& 0& 1& 0& 0  &  2  &  0  &\dots  \\
\dots & 0& 0& 1& 0& 0& 2& 0  &  0  &  1 &\dots \\
\dots & 1& 0& 0& 2& 0& 0& 1  &  0  &  0 &\dots \\
\dots & 0& 2& 0& 0& 1& 0& 0  & 2   &  0 &\dots  \\
\dots & 0& 0& 1& 0& 0& 2& 0  & 0   &  1 &\dots  \\
 \dots &\dots &\dots &\dots &\dots &\dots &\dots &\dots&\dots &\dots &\dots  
\end{array}\right]  \,.
\label{vzorec1}
\end{equation}
Moreover, it is easy to write explicit formula for the values, namely
$$f_1(i,j)  =
        \left \{ \begin{array}{ll}
             0,                  &  \quad i  \not\equiv j\pmod 3\\
             \DR{2-} i\bmod 2,       & \quad i    \equiv j\pmod 3\\
             \end{array} \right.  \,.
$$
%
%
The alternative is to define  a 2RDF as
$$f_2(i,j)  =
        \left \{ \begin{array}{ll}
             0,                  &  \quad i   \not\equiv j\pmod 3\\
            \DR{2-} j\bmod 2,       & \quad i  \equiv j\pmod 3\\
             \end{array} \right. \,,
$$ 
which results in the pattern 
%
\begin{equation}
\left[\begin{array}{ccccccccccc}
 \dots &\dots &\dots &\dots &\dots &\dots &\dots &\dots&\dots &\dots&\dots  \\
\dots & 1& 0& 0& 1& 0& 0& 1  &  0  &  0 &\dots  \\
\dots & 0& 2& 0& 0& 2& 0& 0  &  2  &  0  &\dots  \\
\dots & 0& 0& 1& 0& 0& 1& 0  &  0  &  1 &\dots \\
\dots & 2& 0& 0& 2& 0& 0& 2  &  0  &  0 &\dots \\
\dots & 0& 1& 0& 0& 1& 0& 0  & 1   &  0 &\dots  \\
\dots & 0& 0& 2& 0& 0& 2& 0  & 0   &  2 &\dots  \\
 \dots &\dots &\dots &\dots &\dots &\dots &\dots &\dots&\dots &\dots &\dots  
\end{array}\right]\, .
\label{vzorec2}
\end{equation}
 It is easy to see that the first pattern results in 2RDF's with
$\displaystyle \gamma_{r2}(C_m \Box C_n)= \frac{mn}{3}$, if $m \equiv 0 \pmod 3$, $n \equiv 0 \pmod 6$.
The second pattern provides 2RDF's with
$\displaystyle \gamma_{r2}(C_m \Box C_n)= \frac{mn}{3}$ if $m \equiv 0 \pmod 6$, $n \equiv 0 \pmod 3$.
Note that $m \ge 6$ is required for the second pattern, while the first pattern can be applied if $m\ge 3$.

\medskip\noindent
{\bf Remark.} It is worth noting that in both cases we have $s_i = \frac{m}{3}$.

\medskip
Now we outline  constructions that directly imply some upper bounds.

 \begin{prop}\label{Upper02} 
Let $m \equiv 2 \pmod 3$ and  $n \equiv 0 \pmod 6$.
Write $m = 3k+2$. 
Then 
$$\gamma_{r2}(C_{m} \Box C_n)      \leq    kn +  n .
$$ 
\end{prop}

\begin{proof}  
First, we provide a 2RDF proving  that 
$\gamma_{r2}(C_{5} \Box C_n)    \leq  2n$.
Start with the pattern (\ref{vzorec2}), 
use the first six rows and 
replace the 2nd and 3rd row with the union of them.
$$\left[\begin{array}{ccccccccccc} 
\dots & 1& 0& 0& 1& 0& 0& 1  &  0  &  0 &\dots  \\
\dots &{\bf  0}& {\bf 2}&{\bf  1}& {\bf 0}& {\bf 2}& {\bf 1}& {\bf 0}&  {\bf 2}& {\bf  1} &\dots \\ 
\dots & 2& 0& 0& 2& 0& 0& 2  &  0  &  0 &\dots \\
\dots & 0& 1& 0& 0& 1& 0& 0  & 1   &  0 &\dots  \\
\dots & 0& 0& 2& 0& 0& 2& 0  & 0   &  2 &\dots  \\ 
\end{array}\right]
$$
Is it obvious that the same construction gives   2RDF's proving  that 
$$\gamma_{r2}(C_{3k+2} \Box C_n)    \leq  kn +n,
$$
 as claimed.
\end{proof}
 

 \begin{prop}\label{Upper01} 
Let $m \equiv 1 \pmod 3$ and  $n \equiv 0 \pmod 6$.
Write $m = 3k+1$. 
Then 
$$\gamma_{r2}(C_{m} \Box C_n)      \leq    kn +  n .  
$$ 
\end{prop}

\begin{proof} 
First, we provide a 2RDF proving  that 
$\gamma_{r2}(C_{4} \Box C_n)    \leq  2n$.
Start with the pattern (\ref{vzorec2}), 
use the first six rows,
replace the 2nd and 3rd row with the union of them,  and
replace the 4th and 5th  row with the  union of them.\\
$$\left[\begin{array}{ccccccccccc} 
\dots & 1& 0& 0& 1& 0& 0& 1  &  0  &  0 &\dots  \\
\dots &{\bf  0}& {\bf 2}&{\bf  1}& {\bf 0}& {\bf 2}& {\bf 1}& {\bf 0}&  {\bf 2}& {\bf  1} &\dots \\ 
\dots &{\bf 2}& {\bf 1}&{\bf  0}& {\bf 2}& {\bf 1}&{\bf  0}& {\bf 2}& {\bf 1}&{\bf  0}  &\dots \\  
\dots & 0& 0& 2& 0& 0& 2& 0  & 0   &  2 &\dots  \\ 
\end{array}\right]
$$
Is it obvious that the same construction gives
 2RDF's proving  that 
$$\gamma_{r2}(C_{3k+1} \Box C_n)    \leq  kn +n,
$$
 as claimed.
  \end{proof}

To summarize, we can combine the Propositions \ref{Upper02} and  \ref{Upper01}  with Theorem \ref{Exact} to obtain 
 \begin{prop}\label{UpperMod6} 
Let   $n \equiv 0 \pmod 6$  and $m \ge 3$.  
Then $\gamma_{r2}(C_{m} \Box C_n)      \leq    \lceil \frac{m}{3}   \rceil n  \, .$
\end{prop}
  


The next propositions  provide  general upper bounds for the cases when $n \not\equiv 0 \pmod 6$.
Below we provide constructions based on the previously studied 2RDF  
for each possible reminder $b =0,1,2,3,4,5$  where   $n \equiv   b \pmod 6$.
We start with the case $m\equiv 0 \pmod 3$.

\begin{prop}\label{UpperGeneric}   
Let  $m\ge 3$,  $m \equiv 0 \pmod 3$,  and $n\ge 6$,  
$n \equiv  b \pmod 6$.
Hence $n=6a+b$ for some integer $a \ge 0$.
Then 
 \begin{enumerate}
\item[a)] 
      if $b=5$ then  $\gamma_{r2}(C_{m} \Box C_n)  \le   \gamma_{r2}(C_{m} \Box C_{6a})   +   2m  = \frac{m}{3}(n+1)$, 
\item[b)]   
      if $b=4$ then  $\gamma_{r2}(C_{m} \Box C_n)  \le   \gamma_{r2}(C_{m} \Box C_{6a})   +   2m  = \frac{m}{3}(n+2)$, 
\item[c)]   
      if $b=3$ then  $\gamma_{r2}(C_{m} \Box C_n)  \le  \gamma_{r2}(C_{m} \Box C_{6a})   +   4 \frac{m}{3}  = \frac{m}{3}(n+1)$, 
\item[d)]   
      if $b=2$ then  $\gamma_{r2}(C_{m} \Box C_n)  \le  \gamma_{r2}(C_{m} \Box C_{6a})  +   m  = \frac{m}{3}(n+1)$, 
\item[e)]   
      if $b=1$ then  $\gamma_{r2}(C_{m} \Box C_n)  \le  \gamma_{r2}(C_{m} \Box C_{6a})  + 2 \frac{m}{3}  = \frac{m}{3}(n+1)$ .  
\end{enumerate}    
\end{prop}

\begin{proof}    
In the following we give explicit constructions for the case $m=6 = 2\times 3$ and various $n$.
It is obvious that in general we can simply repeat the pattern of three consecutive rows.
The weight of a column is $\frac{m}{3}$, hence the bounds given in proposition.
 \begin{enumerate}
\item[a)]
 if $b=5$, then replace two columns of the 2RDF $(C_{m} \Box C_{6a+6})$ by their union and observe that the table gives a 2RDF of $(C_{m} \Box C_{6a+5})$.

\medskip

\noindent{\footnotesize
$$\left[\begin{array}{ccccccccccccccc}
\dots & 1& 0& 0& 2& 0& 0 &|&  1&  0& 0& |& 2& 0& 0   \\
\dots & 0& 2& 0& 0& 1& 0 &|&  0& 2& 0& |&0& 1& 0  \\
\dots & 0& 0& 1& 0& 0& 2 &|&  0& 0& 1& |&0& 0& 2  \\
\dots & 1& 0& 0& 2& 0& 0 &|&  1& 0& 0& |&2& 0& 0  \\
\dots & 0& 2& 0& 0& 1& 0 &|&  0& 2&  0& |&0& 1& 0 \\
\dots & 0& 0& 1& 0& 0& 2 &|&  0& 0&  1& |&0& 0& 2  \\
\end{array}\right] 
{\Large\to}$$
$$\left[\begin{array}{ccccccccccccccc}
\dots & 1& 0& 0& 2& 0& 0 &|&  1& 0&  0&    |&  {\bf 2 }&0   \\
\dots & 0& 2& 0& 0& 1& 0 &|&  0& 2& 0&    |& {\bf 1} & 0   \\
\dots & 0& 0& 1& 0& 0& 2 &|&  0& 0& 1&    |&        0  & 2  \\
\dots & 1& 0& 0& 2& 0& 0 &|& 1& 0& 0&    |&  {\bf 2}& 0  \\
\dots & 0& 2& 0& 0& 1& 0 &|&  0& 2&  0&    |& {\bf 1} & 0 \\
\dots & 0& 0& 1& 0& 0& 2 &|&  0& 0&  1&    |&        0  & 2  \\
\end{array}\right].
$$
}

\medskip
So if we look at the last 6 columns, which have shrunk to 5 columns, we see that the number of colors used does not change.
If instead of $m=6$ we consider$m= 3k$,
three rows, e.g. rows 4-6, are repeated $(k-2)$ times and the same construction is applied. The last 6 columns therefore contain $6\times k =2m$ colors.

In the remaining cases, we only give the tables containing the constructions that alter the rightmost columns
(in the tables $m=6$ is chosen).

\medskip
\item[b)]
 if $b=4$, then take (for example) the last four columns and replace them with two columns, each of which is the union of two columns.

\medskip

\noindent{\footnotesize
$$\left[\begin{array}{ccccccccccccccc}
\dots & 1& 0& 0& 2& 0& 0 &|& 1&  0& |& 0& 2& 0& 0   \\
\dots & 0& 2& 0& 0& 1& 0 &|&  0& 2& |&0& 0& 1& 0  \\
\dots & 0& 0& 1& 0& 0& 2 &|&  0& 0& |& 1& 0& 0& 2  \\
\dots & 1& 0& 0& 2& 0& 0 &|&  1& 0& |& 0& 2& 0& 0  \\
\dots & 0& 2& 0& 0& 1& 0 &|&  0& 2&  |& 0& 0& 1& 0 \\
\dots & 0& 0& 1& 0& 0& 2 &|&  0& 0&  |& 1& 0& 0& 2  \\
\end{array}\right]
{\Large\to}$$
$$\left[\begin{array}{ccccccccccccccc}
\dots & 1& 0& 0& 2& 0& 0 &|&  1& 0&     |&         0& 2    \\
\dots & 0& 2& 0& 0& 1& 0 &|&  0& 2&    |& {\bf 1}& 0   \\
\dots & 0& 0& 1& 0& 0& 2 &|&  0& 0&    |&         1& {\bf 2}  \\
\dots & 1& 0& 0& 2& 0& 0 &|&  1& 0&    |&         0& 2   \\
\dots & 0& 2& 0& 0& 1& 0 &|&  0& 2&     |& {\bf 1}& 0 \\
\dots & 0& 0& 1& 0& 0& 2 &|&  0& 0&     |&         1& {\bf 2}  \\
\end{array}\right].
$$
}

\medskip

\item[c)]   
     if $b=3$, then take (for example) the last six columns and replace them with three columns, as follows

\medskip

\noindent{\footnotesize
$$\left[\begin{array}{ccccccccccccccc}
\dots & 1& 0& 0& 2& 0& 0 &|&  1&    0& 0& 2& 0& 0   \\
\dots & 0& 2& 0& 0& 1& 0  &|&  0 &  2&  0& 0& 1& 0  \\
\dots & 0& 0& 1& 0& 0& 2  &|&  0&   0&  1& 0& 0& 2  \\
\dots & 1& 0& 0& 2& 0& 0  &|&  1&   0&  0& 2& 0& 0  \\
\dots & 0& 2& 0& 0& 1& 0 &|&  0&  2&  0& 0& 1& 0 \\
\dots & 0& 0& 1& 0& 0& 2 &|&  0&  0&  1& 0& 0& 2  \\
\end{array}\right]
{\Large\to}$$
$$\left[\begin{array}{ccccccccccccccc}
\dots & 1& 0& 0& 2& 0& 0 &|&  1&     0         &                0  \\
\dots & 0& 2& 0& 0& 1& 0 &|&  0&    \{{\bf 1,2}\} &           0   \\
\dots & 0& 0& 1& 0& 0& 2 &|&  0&      0       &                2   \\
\dots & 1& 0& 0& 2& 0& 0 &|&  1&    0        &                0    \\
\dots & 0& 2& 0& 0& 1& 0 &|&  0&   \{ {\bf 1,2\} }&             0  \\
\dots & 0& 0& 1& 0& 0& 2 &|&  0&    0          &             2  \\
\end{array}\right] ~.
$$
}

\medskip
Note that the 2RDF in this case is not a singleton 2RDF.
A singleton 2RDF either assigns a singleton to the empty set \cite{Erves2021-3}.
We do not know whether there is a singleton 2RDF with the same weight.

 
\medskip

\item[d)]   
     if $b=2$, then take (for example) the last six columns and replace them with two columns, as follows

\medskip

\noindent{\footnotesize
$$\left[\begin{array}{cccccccccccccc}
\dots & 1& 0& 0& 2& 0& 0&|&  1&     0& 0& 2& 0& 0   \\
\dots & 0& 2& 0& 0& 1& 0&|&  0 &  2&  0& 0& 1& 0  \\
\dots & 0& 0& 1& 0& 0& 2&|&  0&   0&  1& 0& 0& 2  \\
\dots & 1& 0& 0& 2& 0& 0&|&  1&   0&  0& 2& 0& 0  \\
\dots & 0& 2& 0& 0& 1& 0&|&  0&    2&  0& 0& 1& 0 \\
\dots & 0& 0& 1& 0& 0& 2&|&  0&   0&  1& 0& 0& 2  \\
\end{array}\right]
{\Large\to}$$
%
$$\left[\begin{array}{cccccccccccccc}
\dots & 1& 0& 0& 2& 0& 0 &|&  {1} &    0       \\
\dots & 0& 2& 0& 0& 1& 0 &|&    0  & {\bf 2}  \\
\dots & 0& 0& 1& 0& 0& 2 &|&     0 &  {2}  \\
\dots & 1& 0& 0& 2& 0& 0 &|&   {1}&    0      \\
\dots & 0& 2& 0& 0& 1& 0 &|&    0  &  {\bf 2} \\
\dots & 0& 0& 1& 0& 0& 2 &|&     0 &   {2}  \\
\end{array}\right].
$$
}
 
\medskip
\item[e)]   
      if $b=1$, then   replace the last six columns with an altered column.

\medskip
\noindent{\footnotesize
$$\left[\begin{array}{cccccccccccccc}
\dots & 1& 0& 0& 2& 0& 0  &|& 1&     0& 0& 2& 0& 0   \\
\dots & 0& 2& 0& 0& 1& 0  &|& 0 &  2&  0& 0& 1& 0  \\
\dots & 0& 0& 1& 0& 0& 2  &|& 0&   0&  1& 0& 0& 2  \\
\dots & 1& 0& 0& 2& 0& 0  &|& 1&   0&  0& 2& 0& 0  \\
\dots & 0& 2& 0& 0& 1& 0  &|& 0&    2&  0& 0& 1& 0 \\
\dots & 0& 0& 1& 0& 0& 2  &|& 0&   0&  1& 0& 0& 2  \\
\end{array}\right]
{\Large\to}$$
$$\left[\begin{array}{cccccccccccccc}
\dots & 1& 0& 0& 2& 0& 0  &|&     {\bf 1}   \\
\dots & 0& 2& 0& 0& 1& 0  &|&         0      \\
\dots & 0& 0& 1& 0& 0& 2  &|&   {\bf 2}   \\
\dots & 1& 0& 0& 2& 0& 0  &|&   {\bf 1}    \\
\dots & 0& 2& 0& 0& 1& 0  &|&       0      \\
\dots & 0& 0& 1& 0& 0& 2  &|&   {\bf 2}  \\
\end{array}\right].
$$
}
\end{enumerate}    
  \end{proof}
Now we generalize Proposition \ref{UpperGeneric}   to arbitrary $m$.

 \SB{\begin{prop}\label{UpperGeneric2}   
Let  $m\ge 3$ and $n\ge 6$.
      If $n\equiv 4 \pmod 6$, then $\gamma_{r2}(C_{m} \Box C_n)  \le    \lceil \frac{m}{3}\rceil (n+2) $.
Otherwise, $\gamma_{r2}(C_{m} \Box C_n)  \le    \lceil \frac{m}{3}\rceil (n+1) $.     
\end{prop}

 \begin{proof} (sketch)
The bounds are obtained by constructions that combine the ideas from Propositions \ref{Upper02}, \ref{Upper01} and \ref{UpperGeneric}.
The main idea is the following.
Start with $C_{\tilde m} \Box C_n$ where $\tilde m = 3 \lceil \frac{m}{3} \rceil$.
Note that there is at most $ \lceil \frac{m}{3} \rceil$ colors in each column.
Apply the constructions as in the proofs of Propositions \ref{Upper02}, \ref{Upper01} and \ref{UpperGeneric}.
Recall that in each of these constructions some columns are deleted and we replace one or two rows by unions of two rows.
The total weight is preserved in this way, so the proposition holds.
  \end{proof}}

The upper bounds provided  in Propositions \ref{UpperGeneric}   and \ref{UpperGeneric2} have a similar form, and can be written in a condensed way as follows. 
\begin{cor}\label{Cor2}
Let  $m \geq 3$ and $n \ge 6$.   Then
$$\displaystyle    
 \gamma_{r2}(C_{m} \Box C_n)  \le  \left\lceil \frac{m}{3}\right\rceil (n    + \beta)  \,,    
\quad \textrm{ where }   \quad
 \displaystyle\beta =
        \left \{ \begin{array}{ll}
             0,                       &   n \equiv 0 \pmod 6\\
             1,                       &   n \equiv 1,2,3,5 \pmod 6\\ 
             2,                	     &   n \equiv 4 \pmod 6\\ 
             \end{array} \right. \,.
$$
\end{cor}

The construction used in Propositions \ref{UpperGeneric}, \ref{UpperGeneric2}, and Corollary  \ref{Cor2} are  based on the basic assignment (\ref{vzorec1}).
Constructions based on (\ref{vzorec2}) can be used in a similar way and result in slightly different upper bounds.

 \begin{prop}\label{UpperGeneric3}
Let $m \geq 6$ and $n \ge 3$. Then
$$\displaystyle
 \gamma_{r2}(C_{m} \Box C_n) \le \left\lceil \frac{n}{3}\right\rceil (m + \gamma) \,,
\quad \textrm{ where } \quad
 \displaystyle\gamma =
 \left \{ \begin{array}{ll}
 0, & m \equiv 0 \pmod 6\\
 1, & m \equiv 5 \pmod 6\\
 2, & m \equiv 1,2,3,4 \pmod 6\\
 \end{array} \right. \,.
$$

\end{prop}

 \begin{proof}
We give only a brief outline of the proof and omit the detailed arguments, because the ideas are analogous to those previously elaborated in the proofs of Propositions \ref{UpperGeneric}, \ref{UpperGeneric2} and Corollary \ref{Cor2},

Recall first that for $m\equiv 0 \bmod 6$ and $n \equiv 0 \bmod 3$ Pattern (\ref{vzorec2}) returns a 2RDF with weight $\frac{mn}{3}$.

\SB{Let us now assume that $n \equiv 0 \bmod 3$ and let $m \equiv d \bmod 6$.
We claim that if $d=5$ then $\gamma_{r2}(C_{m} \Box C_n) \le \lceil \frac{n}{3}\rceil (m+1) $,
otherwise, $\gamma_{r2}(C_{m} \Box C_n) \le \lceil \frac{n}{3}\rceil (m+2) $.}  
If $d=5$ then one row is deleted, and the colors of the deleted row are given to one neighboring rows.
Formally, row $m$ is defined as a union of the rows $m$ and $m+1$ of the pattern.
In any other  case, a 2RDF is obtained by deleting some rows and replacing rows 1 and $m$ with unions.

We have thus seen that the cases $n \not\equiv 0 \bmod 3$ can be handled by deleting one or two columns in the pattern.
The colors of the deleted column(s) are then used to complete the assignment of columns 1 and $n$.
And we have the upper bound as claimed.
 \end{proof}

It seems obvious that the two upper bounds are not equivalent. Now we compare them more closely. 
To this end we write 
\begin{eqnarray}
\UBi(m,n) &=&  \left\lceil \frac{m}{3}\right\rceil (n  + \beta)  \nonumber\\ 
               &=&  \frac{1}{3}(m+a) (n + \beta)  =\frac{1}{3}mn  +\frac{1}{3}an  +\frac{1}{3}\beta m   +\frac{1}{3} a \beta  \\
\UB2(m,n)  &=& \left\lceil \frac{n}{3}\right\rceil (m + \gamma)  \nonumber\\ 
               &=&   \frac{1}{3} (n+c)(m+ \gamma ) = \frac{1}{3}mn +\frac{1}{3}\gamma n +\frac{1}{3}c m +\frac{1}{3} \gamma c
 \end{eqnarray}
where 
\begin{eqnarray*}
	\quad \displaystyle  a  &=&
        \left \{ \begin{array}{ll}
             0,                       &  m \equiv 0 \pmod 3\\
             1,                       &   m \equiv 2 \pmod 3\\ 
             2,                	  &   m \equiv 1 \pmod 3\\ 
             \end{array} \right. \,
,    \quad 
\displaystyle\beta =
        \left \{ \begin{array}{ll}
             0,                       &   n \equiv 0 \pmod 6\\
             1,                       &   n \equiv 1,2,3,5 \pmod 6\\ 
             2,                	  &   n \equiv 4 \pmod 6\\ 
             \end{array} \right. \, , 
   \\
 \displaystyle\gamma &=&
        \left \{ \begin{array}{ll}
             0,                       &   n \equiv 0 \pmod 6\\
             1,                       &   n \equiv  5 \pmod 6\\ 
             2,                       &   n \equiv 1,2,3,4 \pmod 6\\  
             \end{array} \right. \, 
,    \textrm{and~~~} 
\displaystyle c  =
        \left \{ \begin{array}{ll}
             0,                       &  n\equiv 0 \pmod 3\\
             1,                       &   n \equiv 2 \pmod 3\\ 
             2,                	  &   n \equiv 1 \pmod 3\\ 
             \end{array} \right.  \,.
\end{eqnarray*}
%
%
%
Note that both \UBi and \UB2  are of the form $\frac{1}{3}mn  +\frac{1}{3}   (x,y,z)(n,m,1)$, 
and let us write the values of $(x,y,z)$ in two tables for easier comparison.

\begin{center}
	\begin{tabular}{|c|c|ccc|ccc| } 
	\hline  
	$\UBi(m,n)$      & $n \bmod 6$          & 0& 1& 2& 3& 4&  5 \\   
	\hline 
	$m \bmod 6$ & $a \backslash \beta$ & 0& 1& 1& 1& 2&  1 \\ 
	\hline 
	0  & 0                                     & (0,0,0) & (0,1,0) & (0,1,0) & (0,1,0) & (0,2,0) &  (0,1,0)  \\ 
	1  & 2                                     & (2,0,0) & (2,1,2) & (2,1,2) & (2,1,2) & (2,2,4) &  (2,1,2)    \\ 
	2 &  1                                     & (1,0,0) & (1,1,1) & (1,1,1) & (1,1,1)  & (1,2,2) & (1,1,1)   \\ 
	\hline
	3  & 0                                     & (0,0,0) & (0,1,0) & (0,1,0) & (0,1,0) & (0,2,0) &  (0,1,0)  \\ 
	4  & 2                                     & (2,0,0) & (2,1,2) & (2,1,2) & (2,1,2) & (2,2,4) &  (2,1,2)    \\ 
	5 &  1                                     & (1,0,0) & (1,1,1) & (1,1,1) & (1,1,1)  & (1,2,2) & (1,1,1)   \\  
	\hline
\end{tabular} 
\end{center}


\medskip

\begin{center}
	\begin{tabular}{|c|c|ccc|ccc| } 
		\hline  
$\UB2(m,n)$        & $n \bmod 6$          & 0& 1& 2& 3& 4&  5 \\  
\hline 
$m \bmod 6$ & $\gamma\backslash c$ & 0& 2& 1& 0& 2&  1 \\ 
\hline 
0  & 0                                     & (0,0,0) & (0,2,0) & (0,1,0) & (0,0,0) & (0,2,0) &  (0,1,0)  \\ 
1  & 2                                     & (2,0,0) & (2,2,4) & (2,1,2) & (2,0,0) & (2,2,4) &  (2,1,2)    \\ 
2 &  2                                    & (2,0,0) & (2,2,4) & (2,1,2) & (2,0,0) & (2,2,4) &  (2,1,2)    \\  
\hline
3  & 2                                    & (2,0,0) & (2,2,4) & (2,1,2) & (2,0,0) & (2,2,4) &  (2,1,2)    \\  
4  & 2                                    & (2,0,0) & (2,2,4) & (2,1,2) & (2,0,0) & (2,2,4) &  (2,1,2)    \\ 
5 &  1                                     & (1,0,0) & (1,2,2) & (1,1,1) & (1.0,0)  & (1,2,2) & (1,1,1)   \\  
\hline
\end{tabular}
\end{center} 
Comparison is summarized in the next table. In \DR{fourteen} cases $\UBi < \UB2$, 
in other words the first pattern gives rise  a better 2RDF.
In four cases,  $\UB2 < \UBi$. 
Note that in two cases, the triples are no comparable.
In particular, when 
$m \equiv 2 \pmod 6$ and $n \equiv 3 \pmod 6$ 
we have 
$$ \UBi =  \frac{1}{3}mn  +\frac{1}{3}(n+m+1 )    <> \UB2 =   \frac{1}{3}mn  +\frac{1}{3} 2m   
$$
and hence 
$$\UBi  >= <  \UB2 \quad \iff \quad n+1  >= <  m \,.
$$
Similarly, when 
$m \equiv 3 \pmod 6$ and $n \equiv 3 \pmod 6$, 
$$ \UBi =  \frac{1}{3}mn  +\frac{1}{3} n   <> \UB2 =   \frac{1}{3}mn  +\frac{1}{3} 2m   
$$
and hence 
$$\UBi  >= <  \UB2 \quad \iff \quad n  >= <  2m \,.
$$
We summarize the observations in the next table.
\begin{center}
	\begin{tabular}{|c|ccc|ccc| } 
\hline 
                $m \backslash n  \bmod 6$          & 0& 1& 2& 3& 4&  5 \\   
\hline 
0  &                                        =  &   $\UBi(m,n)$   &   =  &  $\UB2(m,n)$  &   =  &   =  \\ 
1  &                                        =  &   $\UBi(m,n)$   &   =  &  $\UB2(m,n)$  &   =  &   =    \\ 
2 &                          \DR{ $\UBi(m,n)$}   &  $\UBi(m,n)$   &  $\UBi(m,n)$  & $>= <$&  $\UBi(m,n)$  &$\UBi(m,n)$   \\  
\hline
3  &                         \DR{ $\UBi(m,n)$}   &  $\UBi(m,n)$  & $\UBi(m,n)$ &  $>= <$ &  $\UBi(m,n)$  &$\UBi(m,n)$     \\   
4  &                                      =  &  $\UBi(m,n)$  &   =  & $\UB2(m,n)$ &   =  &   =    \\ 
5 &                                       =  &  $\UBi(m,n)$   &   =  & $\UB2(m,n)$ &   =  &   =    \\ 
\hline
\end{tabular}
\end{center}

Finally, we recall that the Cartesian product is commutative, $C_m \Box C_n  \simeq C_n \Box C_m$. 
Therefore, the best upper bound for $\gamma_{r2}(C_{m} \Box C_n)$ is
based on the constructions considered here and is the minimum 
of the bounds  
$\UBi(m,n),   \UB2(m,n),    \UBi(n,m)$,  and $\UB2(n,m)$. The results are written below.
%
%
%
\DR{
\begin{center}
	\begin{tabular}{|c|ccc|ccc| } 
		\hline 
               $m \backslash n  \bmod 6$            & 0& 1& 2& 3& 4&  5 \\   
\hline 
0  &                                        =  &   $\UBi(m,n)$    &   $\UBi(m,n)$  &  $\UB2(m,n)$           &   =  &   =  \\ 
1  &                                            &   $*$    &   $\UBi(n,m)$   &  $ \UBi(n,m) $ & $ \UBi(n,m) $ & $   \UBi(n,m) $    \\ 
2 &                                             &           & $ \UBi(m,n) $ &  $ \UBi(n,m) $ &  $\UBi(m,n)$          &$\UBi(m,n)$     \\  
\hline
3  &                                                                             &   &   & $*$ &  $\UBi(m,n) $&$\UBi(m,n)$\\   
4  &                                                                             &   &   &                                           &   =  &   =    \\ 
5 &                                                                              &   &   &                                            &      &   =    \\ 
\hline
\end{tabular}
\end{center}
where 
$*=\min\{\UBi(m,n),\UBi(n,m)\}$ . 

\medskip
Explicitly, the best upper bounds for  $\gamma_{r2}(C_{m} \Box C_n)$  are of the form $$\frac{1}{3}mn +\frac{1}{3}(n,m,1)(x,y,z)$$ with values of $(x,y,z)$ from the following table.
\begin{center}
	\begin{tabular}{|c|ccc|ccc| } 
\hline 
               $ m \backslash n  \bmod 6  $        & 0& 1& 2& 3& 4&  5 \\   
\hline  
0                                  & (0,0,0) & (0,1,0) & (0,1,0) & (0,0,0) & (0,2,0) &  (0,1,0)  \\ 
1                                  &           & (2,1,2) or (1,2,2) & (1,1,1) & (1,0,0) & (1,2,2) &  (0,1,0)    \\ 
2                                  &           &            & (1,1,1) & (1,0,0) & (1,2,2) & (1,1,1)   \\ 
\hline
3                                &            &            &             & (0,1,0) or (1,0,0)  & (0,2,0) &  (0,1,0)  \\ 
4                                &            &            &             &             & (2,2,4) &  (2,1,2)    \\ 
5                                &            &            &             &             &            & (1,1,1)   \\  
\hline
\end{tabular}
\end{center} 

The bounds can be summarized as follows.
\begin{cor} \label{CorTable}
	Let $m\ge 6$ and $n\ge 6$. As the Cartesian product is commutative, we can assume $m \ge n$. Then
	$$\displaystyle    
	\gamma_{r2}(C_{m} \Box C_n)  \le  \frac{1}{3} mn + \frac{1}{3} \delta  \,, $$   
	where $\delta$ can be read from the Table below.
	\begin{center}
		\begin{tabular}{|c|ccc|ccc| } 
			\hline 
			$ m \backslash n$     &            &            &             &             &  &            \\
			$   \bmod 6  $        & $0$& $1$& $2$& $3$& $4$&  $5$ \\     
			\hline  
			$0$                                  & $0$ & $m$ & $m$ & $0$ & $2m$ &  $m$  \\ 
			                                       &           &\JZ{$\min $}  &            &            &    &   \\ 
			$1$                                  &           & \JZ{\small $\{n+2m+2, $}
										 & $n+m+1$ & $n$ & $n+2m+2$ & $ m $   \\ 
			                                       &           & \JZ{\small$ 2n+m+2\}$}   &            &              &    &     \\ 
			$2$                                  &           &            & $n+m+1$ &$ n $& $n+2m+2$& $n+m+1$   \\ 
			\hline
			$3$                                &            &            &             & \JZ{$\min\{m,n\}$}  & $2m$ &  $m$  \\ 
			$4$                                &            &            &             &             & $2n+2m+4$ &  $2n+m+2$    \\ 
			$5$                                &            &            &             &             &            & $n+m+1 $  \\  
			\hline
		\end{tabular}
	\end{center}
	\end{cor}
}

\section{Conclusions and future work}
%
%
%
%
%

We have provided lower and upper bounds for the 2-rainbow domination number of $C_m \Box C_n$ with a gap of at most $\frac{1}{3} (2m+2n+4)$.
The proof of the lower bound is based on a discharging argument and seems to be close to the best possible in most cases.
The upper bound, on the other hand, is based on two constructions that are quite rough in some cases, and we believe that it can be improved by carefully analyzing special cases.
\JZ{We conjecture that the lower bounds differ from the exact values by at most one constant, which depends on $m$ and is independent of $n$}

At least for examples with small $m$, we claim that it is possible to avoid the tedious analysis by applying an algebraic method that can be used for various graph invariants including the domination type problems
 \cite{KlavzarZerovnik,PavlicKragujevac,PavlicArsComb,GabrovsekCEJOR}.
Such a research task remains a challenge for future work.

Another interesting line of research, which is a natural extension of this study, is a generalization of the results presented here to graph bundles, a natural generalization of graph products \cite{ST_P_V}.

%
%

\end{document}